\documentclass[11pt]{amsart}
\usepackage{xcolor}
\usepackage{amscd}
\usepackage{amsmath,amsfonts}
\usepackage{color}
\usepackage{graphicx}
\usepackage{float}
\usepackage[utf8]{inputenc}

\newcommand{\N}{\mathbb{N}}

\newcommand{\set}[1]{\left\{#1\right\}}
\newcommand{\norm}[2]{\| #1 \|_{#2}}

\newenvironment{frcseries}{\fontfamily{frc}\selectfont}{}

\newcommand{\Rg}{\mathcal R}

\newtheorem{theorem}{Theorem}

\newtheorem{proposition}{Proposition}

\newtheorem{lemma}{Lemma}
\newtheorem{definition}{Definition}

\newtheorem{remark}{Remark}

\title[Degree of ill-posedness governed by the Ces\`{a}ro operator]{The degree of ill-posedness for some composition
governed by the Ces\`{a}ro operator}

\author{Yu Deng}
\address{Faculty of Mathematics, Chemnitz University of Technology, 09107 Chemnitz,  Germany}
\email{yu.deng@math.tu-chemnitz.de}
\author{Hans-J\"urgen Fischer}
\address{01159 Dresden, Germany}
\email{hans.fisher@online.de}
\author{Bernd Hofmann}
\address{Faculty of Mathematics, Chemnitz University of Technology, 09107 Chemnitz,  Germany}
\email{hofmannb@mathematik.tu-chemnitz.de}

\date{\today\\Corresponding author: B.~Hofmann}
\begin{document}
\begin{abstract}
  In this article, we consider the singular value asymptotics of compositions of compact
  linear operators mapping in the real Hilbert space of quadratically
  integrable functions over the unit interval. Specifically, the
  composition is given by the compact simple integration operator followed by
  the non-compact Ces\`{a}ro operator possessing a non-closed range.
  We show that the degree of ill-posedness of that composition is {\it two}, which means that the Ces\`{a}ro operator increases the degree of ill-posedness
  by the amount of $one$ compared to the simple integration operator.
  \end{abstract}

\maketitle


\medskip


\medskip

\section{Introduction}
\label{sec:intro}

This is a new paper in the series of articles \cite{Freitag05,HW05,HW09} and recently \cite{HM22} that are dealing with the degree of ill-posedness of linear operator equations
\begin{equation} \label{eq:opeq}
A\,x\,=\,y\,,
\end{equation}
where the compact linear operator $A: X \to Y$ is
factorized as
 \begin{equation} \label{eq:factorized}
 \begin{CD}
  A:\; @.  X @> K >> Z  @> N>> Y\,,
  \end{CD}
\end{equation}
for infinite dimensional separable Hilbert spaces $X,Y$ and $Z$. In this context, $A=N \circ K$ denotes the composition of an injective {\it compact} linear operator $K: X \to Z$ and a bounded injective and {\it non-compact}, but not continuously invertible, linear operator $N:~Z \to Y$.
The imposed requirements on $K$ and $N$ imply that the ranges $\Rg(K)$, $\Rg(N)$ and $\Rg(A)$ are infinite dimensional, but non-closed, subspaces of the corresponding Hilbert spaces.
Following the concept of Nashed \cite{Nashed87}, the total equation \eqref{eq:opeq} and the inner linear operator equation
\begin{equation} \label{eq:inner}
K\, x\,=\,z
\end{equation}
are ill-posed of type~II due to the compactness of $A$ and $K$, whereas the outer linear operator equation
\begin{equation} \label{eq:outer}
N\, z\,=\,y
\end{equation}
is ill-posed of type~I, since $N$ is non-compact.

We recall here a definition of the interval and degree of ill-posedness along the lines of \cite{HofTau97}:

\begin{definition} \label{def:degree}
Let $\{\sigma_n(A)\}_{n=1}^\infty$ the non-increasing sequence  of singular values of the injective and compact linear operator $A: X \to Y$,
tending to zero as $n \to \infty$. Based on the well-defined {\sl interval of ill-posedness} introduced as
$$[\underline{\mu}(A),\overline{\mu}(A)]=\left[\liminf \limits _{n \to \infty}
\frac{-\log \sigma_n(A)} {\log n}\,,\,\limsup \limits _{n \to
\infty} \frac{-\log \sigma_n(A)} {\log n}\right]  \subset [0,\infty],$$
we say that the operator $A$, and respectively the associated operator equation \eqref{eq:opeq}, is {\sl ill-posed of degree} $\mu=\mu(A) \in (0,\infty)$
if $\mu=\underline{\mu}(A)=\overline{\mu}(A)$, i.e., if the interval of ill-posedness degenerates into a single point.
\end{definition}

The main objective of the above mentioned article series and of the present study is to learn whether the non-compact operator $N$ can amend the degree of ill-posedness of the compact operator $K$ by such a composition $A=N \circ K$.
Such amendment would be impossible if $N$ were continuously invertible, or in other words if the outer operator equation \eqref{eq:outer} were well-posed.
Since in our setting zero belongs to the spectrum of $N$, the singular value asymptotics of $A$ can differ from that of $K$, but due to the inequality
\begin{equation} \label{eq:grow}
\sigma_n(A) \le \|N\|_{\scriptscriptstyle {\mathcal L}(Z,Y)}\,\sigma_n(K) \qquad (n \in \N)
\end{equation}
only in the sense of a growing decay rate, which means a growing degree of ill-posedness of $A$ compared with $K$. Note that the estimate \eqref{eq:grow} is an immediate consequence
of the Courant–Fischer min-max principle for the characterization of singular values.

In this study, we will focus on one common Hilbert space $X=Y=Z=L^2(0,1)$, the space of quadratically integrable real functions over the unit interval of the real axis.
We consider as operator $K$ the compact {\sl simple integration operator} $J: L^2(0,1) \to L^2(0,1)$ defined as
\begin{equation}\label{eq:J}
[J x](s):=\int_0^s x(t)dt\qquad(0 \le s \le 1)\,,
\end{equation}
where the singular system
$$\{\sigma_n(J);u_n(J);v_n(J)\}_{n=1}^\infty \quad \mbox{with} \quad\ Ju_n(J)=\sigma_n(J)v_n(J)\quad (n \in \N)$$ is of the form
\begin{equation} \label{eq:SVDJ}
  \left\{\frac{2}{(2n-1)\pi},\sqrt{2}\cos\left(n-\frac{1}{2}\right)\pi t \;(0 \le t \le 1);\sqrt{2}\sin\left(n-\frac{1}{2}\right)\pi t \;(0 \le t \le 1)\right\}_{n=1}^\infty\,.
\end{equation}
The asymptotics \footnote{We use the notation $a_n \asymp b_n$ for
sequences of positive numbers $a_n$ and $b_n$ satisfying inequalities  $\underline{c}\,b_n \le a_n \le \overline{c}\,b_n$ with constants $0<\underline{c} \le \overline{c}<\infty$ for sufficiently large $n \in \N$.}
\begin{equation} \label{eq:simple}
\sigma_n(J) \asymp \frac{1}{n}
\end{equation}
shows that the degree of ill-posedness of $J$ in the sense of Definition~\ref{def:degree} is {\it one}.

Let us briefly mention the former results using $J$ as compact operator in such a composition.  In the papers \cite{HW05,HW09}, it was shown that wide classes of bounded non-compact multiplication operators $M: L^2(0,1) \to L^2(0,1)$ defined as
\begin{equation}\label{eq:multioperator}
[Mx](t):=m(t)\,x(t) \qquad (0 \le t \le 1)\,,
\end{equation}
with a multiplier functions $m \in L^\infty(0,1)$ possessing essential zeros, do not amend the singular value asymptotics. This means $\sigma_n(M \circ J) \asymp \sigma_n(J)$ and implies that
the ill-posedness degree of $A=M \circ J$ stays at one. A first example, where the degree of ill-posedness grows, was presented in \cite{HM22} for $A=H \circ J$ with the bounded non-compact Hausdorff operator $H:L^2(0,1) \to \ell^2(\N)$ defined as
\begin{equation} \label{eq:Haus}
[Hz]_j:= \int_0^1 t^{j-1}z(t)dt \qquad (j=1,2,...)\,,
\end{equation}
and we refer for properties of $H$ to the article \cite{Gerth21}.
In \cite[\S 5]{HM22} it could be shown that for some positive constants $\underline{c}$ and $\overline{c}$ the singular values behave as
\begin{equation} \label{eq:altrate}
\exp(-\underline{c}\,n) \le \sigma_n(H \circ J)  \le \frac{\overline{c}}{n^{3/2}}\,.
\end{equation}
Hence, the ill-posedness interval for the composition $A=H \circ J$ is a subset of the interval $\left[\frac{3}{2},\infty\right]$.

\medskip

Now in the present study, we only consider as bounded non-compact and not continuously invertible operator $N$ the {\sl continuous Ces\`{a}ro operator} $C: L^2(0,1) \to L^2(0,1)$ defined as
\begin{equation}\label{eq:C}
[C x](s):=\frac{1}{s}\int_0^s x(t)dt\qquad(0 < s \le 1)\,.
\end{equation}
We refer to \cite{Brown65} and \cite{Lacruz15,Leib73} for properties including boundedness and further discussions concerning this operator $C$. The two properties of $C$, which are most important for the present study, are outlined in the following Lemma~\ref{lem:two}. Its
proof is given in the appendix.

\begin{lemma} \label{lem:two}
The injective bounded linear operator $C: L^2(0,1) \to L^2(0,1)$ from \eqref{eq:C} is non-compact and not continuously invertible, i.e.,~the inverse operator $C^{-1}: \Rg(C) \subset L^2(0,1) \to L^2(0,1)$ is unbounded and hence the range $\Rg(C)$ is not a closed subset of $L^2(0,1)$.
\end{lemma}

Precisely, in our composition the compact simple integration operator $J$ from \eqref{eq:J} is followed by the non-compact Ces\`{a}ro operator $C$ from \eqref{eq:C} as $A:=C \circ J: L^2(0,1) \to L^2(0,1)$. The compact composite operator $A$ can be written explicitly as
\begin{equation}\label{eq:A}
[Ax](s):= \frac{1}{s}\int_0^s (s-t)\, x(t) dt= \int_0^s \frac{s-t}{s}\, x(t) dt  \qquad(0 \le s \le 1)\,.
\end{equation}

The paper is organized as follows. In Section \ref{sec:twofold} the relationship of the composite operator $A$ with the twofold integration operator is presented. In this way, the lower bound of the degree of the ill-posedness of $A$ can be determined. We analyse some properties of the Hilbert-Schmidt operator $A$ in Section \ref{sec:kernel} and one can establish its approximate decay rate  numerically via calculating the eigenvalues of $A^*A$ with symmetric kernel.  Finally, with the aid of a suitable orthonormal basis in $L^2(0,1)$ we are able to identify the degree of ill-posedness of the composite operator $A=C \circ J$ in Section \ref{sec:improved}.

\section{Cross connections to the twofold integration operator}
\label{sec:twofold}

We recall the family of {\sl Riemann-Liouville fractional integral operators} $J^\kappa: L^2(0,1) \to L^2(0,1)$ defined as
\begin{equation} \label{eq:kappa}
[J^\kappa x](s):=\frac{1}{\Gamma(\kappa)}\int_0^s (s-t)^{\kappa-1} x(t) dt  \qquad(0 \le s \le 1)\,.
\end{equation}
For all $\kappa>0$, the linear operators $J^\kappa$ are injective and compact. We know (cf.~\cite{VuGo94} and references therein) that the singular value asymptotics
\begin{equation}\label{eq:asympkappa}
\sigma_n(J^\kappa) \asymp \frac{1}{n^\kappa}
\end{equation}
holds true. The solution of the equation $J^\kappa x=y$ can be seen as the $\kappa$'s fractional derivative of $y$ such that degree of ill-posedness of this equation is $\kappa$ and grows with the order of differentiation.
Besides the simple integration operator $J$ from \eqref{eq:simple} for $\kappa=1$, the {\sl twofold integration operator}
\begin{equation}\label{eq:twofold}
[J^2 x](s):= \int_0^s (s-t) x(t) dt  \qquad(0 \le s \le 1)
\end{equation}
for $\kappa=2$ (see further details in \cite[Section~11.5]{Ramlau20}) plays some prominent role in our study. Obviously, we can write for the composite operator $A$ from \eqref{eq:A} on the one hand
\begin{equation} \label{eq:div}
[Ax](s)=[J^2x](s)/s \qquad (0<s \le 1)\,,
\end{equation}
and on the other hand
\begin{equation} \label{eq:Mcomp}
J^2= M \circ A \quad \mbox{for multiplication operator}\quad [Mx](s)=s\, x(s).
\end{equation}
We mention here that formula \eqref{eq:div} shows that the range $\Rg(A)$ of \linebreak $A: L^2(0,1) \to L^2(0,1)$ is a subset of the space of continuous functions over $[0,1]$. Namely, $\Rg(J^2)$ is a subset of the Sobolev space $H^2(0,1)$, which is continuously embedded in $C^1[0,1]$ and
contains only Lipschitz continuous functions. Thus, $[Ax](s)$ can be  continuously extended to $s \in [0,1]$.
An application of formula \eqref{eq:grow} yields for the singular values
$$\sigma_n(J^2) \le \|M\|_{\scriptscriptstyle \mathcal{L}(L^2(0,1))}\,\sigma_n(A)\le \sigma_n(A) \le \|C\|_{\scriptscriptstyle \mathcal{L}(L^2(0,1))}\,\sigma_n(J) \qquad (n \in \N)\,.$$
Together with \eqref{eq:asympkappa} we obtain that there exist positive constants $K_1$ and $K_2$ such that
\begin{equation} \label{eq:chain1}
\frac{K_1}{n^2} \le \sigma_n(A) \le \frac{K_2}{n}  \qquad (n \in \N)\,.
\end{equation}
Consequently, we know at this point only that the interval of ill-posedness of the operator $A$ from \eqref{eq:A} is a subset of the interval $[1,2]$.
However, taking into account that $A$ is a Hilbert-Schmidt operator, we will be able to improve the order of the upper bound of \eqref{eq:chain1} in Section~\ref{sec:improved}.

\section{Hilbert-Schmidt property and kernel smoothness}
\label{sec:kernel}

As one sees from \eqref{eq:A}, $A=C \circ J: L^2(0,1) \to L^2(0,1)$ is linear Volterra integral operator with quadratically integrable kernel and hence a Hilbert-Schmidt operator with Hilbert-Schmidt norm square
\begin{equation} \label{eq:HS}
\|A\|_{HS}^2  =\int \limits_0^1 \int \limits_0^s \left(\frac{s-t}{s}\right)^2 dt ds = \frac{1}{6}\,.
\end{equation}
Taking into account that the adjoint operator $A^*: L^2(0,1) \to L^2(0,1)$ is of the form $$[A^*y](t)=\int \limits_t^1 \frac{s-t}{s}\, y(s) ds \quad (0 \le t \le 1)\,,$$
we derive the structure of the symmetric kernel $k(s,t)$ of the self-adjoint Fredholm integral operator
$$ [A^*A \,w](t)=\int_0^1 k(t,s)\,w(s) ds \quad (0 \le t, s \le 1)$$
as
\begin{equation} \label{eq:kernel}
\begin{aligned}
k(t,s)&= \int \limits _{\max(t,s)}^1 \left(\frac{\tau-t}{\tau} \right)\,\left(\frac{\tau-s}{\tau} \right)\, d\tau\\
 &= {\scriptstyle \left\{\begin{array}{ll}
1-st-s+t\ln s+s\ln s+t &\quad (0<  t\leq s \le 1)\\
1-st-t+t\ln t+s\ln t+s& \quad (0<s< t \le  1)
\end{array} \right.} \,.
\end{aligned}
\end{equation}
It is well-known that decay rates of the singular values of a compact linear operator grows in general with the smoothness of the kernel. Unfortunately,
the kernel \eqref{eq:kernel} is continuous on the unit square only with the exception of the origin $(0,0)$, where a pole arises. Therefore, usually applied assertions on kernel smoothness (cf.,~e.g., \cite{Chang52,Reade83,Read83}) cannot be exploited to estimate the asymptotics of the
eigenvalues $\lambda_n(A^*A)$ and in the same manner the asymptotics of the singular values $\sigma_n(A)=\sqrt{\lambda_n(A^*A)}$ of $A$. On the other hand, by twice differentiation of the function $[A^*Aw](t)\,(0 \le t \le 1)$ as
\begin{equation*}
\begin{aligned}
[A^*Aw](t)&=\int_t^1\frac{s-t}{s}\int_0^s\frac{s-\tau}{s}w(\tau)d\tau ds\\
[A^*Aw]'(t)&=\int_t^1(-\frac{1}{s})\int_0^s\frac{s-\tau}{s}w(\tau)d\tau ds\\
&=-\int_t^1\int_0^s\frac{s-\tau}{s^2}w(\tau)d\tau ds\\
[A^*Aw]''(t)&=\int_0^t\frac{t-\tau}{t^2}w(\tau)d\tau,\\
\end{aligned}
\end{equation*}
we have an integro-differential final value problem
\begin{equation}\label{eq:BVP}
\int \limits_0^t \frac{t-\tau}{t^2}\,w(\tau) d\tau=\lambda\,w^{\prime\prime}(t)\;\;(0<t<1), \quad w(1)=w^\prime(1)=0 \,.
\end{equation}
Achieving an explicit analytical solution of the eigenvalues $\lambda_n(A^*A)$ and corresponding
eigenfunctions $w \in L^2(0,1)$ seems to be very difficult.

However, we are still able to calculate numerical approximations of the eigenvalues $\lambda_n(A^*A)$ from \eqref{eq:BVP} for small $n$. We have made use of the technique of finite difference discretization with a specific rectangular rule. In this context, the unit interval $[0,1]$ had been divided into $\ell$ partitions with the uniform length $h=1/\ell$. The function values $w(\tau)$ are represented by discrete values $w_i:=w(i*h)\:(i=0,1,\dots,\ell)$. The discrete counterpart of its second derivative $w^{\prime\prime}(t)$ is considered as
$\frac{w_{j+1}-2w_j+w_{j-1}}{h^2}\;(j=1,2,...,\ell-1)$.
The equation \eqref{eq:BVP} can be written in a discrete form as
$$h\sum_{i=0}^{j-1}\frac{(j-i)h}{(jh)^2}w_i=\lambda\frac{w_{j+1}-2w_j+w_{j-1}}{h^2}\qquad j=1,\dots, \ell-1.$$
 with the boundary conditions of as $w_\ell=0$ and $w_{\ell}=w_{\ell-1}$.

Now we need to find those positive values $\lambda$  such that the determinant of the $(\ell+1)\times
(\ell+1)$ square matrix \begin{small}
 \[\left(\begin{array}{lllllllll}
 h^2-\lambda& 2\lambda&-\lambda&0&0&0&0&0&0\\
 \tfrac24 h^2&\tfrac14 h^2-\lambda& 2\lambda&-\lambda&0&0&0&0&0\\
 \tfrac39 h^2&\tfrac29 h^2&\tfrac19 h^2-\lambda& 2\lambda&-\lambda&0&0&0&0\\
 \vdots & \ddots& \ddots& \cdots& \cdots&\cdots&\ddots&\ddots&\vdots\\
 \tfrac{i}{i^2}h^2&\tfrac{i-1}{i^2}h^2&\cdots&\tfrac{1}{i^2}h^2-\lambda&2\lambda&-\lambda&0&0&0\\
 \vdots & \ddots& \ddots& \cdots& \cdots&\cdots&\ddots&\ddots&\vdots\\
 \tfrac{\ell-1}{(\ell-1)^2}h^2&\tfrac{\ell-2}{(\ell-1)^2}h^2&\cdots&\tfrac{2}{(\ell-1)^2}h^2&\tfrac{1}{(\ell-1)^2}h^2-\lambda&2\lambda&-\lambda&0&0\\
 0&0&\cdots&\cdots&0&0&0&0&1\\
 0&0&\cdots&\cdots&0&0&0&1&-1\\
  \end{array}\right)\]
  \end{small}

{\parindent0em  equals} zero, which gives the sequence of discretized eigenvalues $\lambda_n \;(n=1,2,...,\ell+1)$  in non-increasing order. The following Figure \ref{fig:EW} displays in a double logarithmic representation the decay of calculated eigenvalues for the case $\ell=20$.
From that curve an asymptotics of the form $\lambda_n(A^*A)\asymp n^{-4}$ can be predicted, which would coincide with the lower bound of the inequality chain \eqref{eq:chain1}.
\begin{figure}[H]
\begin{center}
\includegraphics[width=8cm]{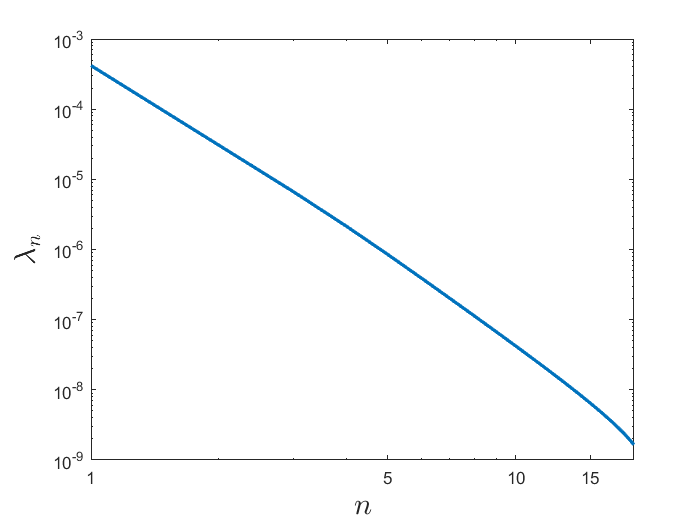}
\caption{Double logarithmic representation of the computationally approximated eigenvalues $\lambda_n(A^*A)$ for small indices $n$.}
\label{fig:EW}
\end{center}
\end{figure}

{\parindent0em This} computational prediction will be confirmed by the analytical study of the subsequent section.

\section{Improved upper bounds}
\label{sec:improved}

To reduce the asymptotics gap, which has been opened by the inequality chain \eqref{eq:chain1}, we reuse the Hilbert-Schmidt operator technique introduced in Section~5 of the recent article \cite{HM22}.
First, we briefly summarize the main ideas of this approach. For the Hilbert-Schmidt operator $A: L^2(0,1) \to L^2(0,1)$ with singular system $\{\sigma_n(A),u_n=u_n(A),v_n=v_n(A)\}_{n=1}^\infty$ we have for the Hilbert-Schmidt norm square
$$ \|A\|_{HS}^2 = \sum \limits_{n=1}^\infty \sigma_n^2(A)\,. $$
We start with the following proposition, which can be found with a complete proof as Proposition~3 in \cite{HM22} when taking into account Theorem~15.5.5 from \cite{Pie78}.

\begin{proposition} \label{pro:HM1}
Let $\{e_i\}_{i=1}^\infty$ denote an arbitrary orthonormal basis in $L^2(0,1)$ and $Q_n$ denote the orthogonal projections onto the $n$-dimensional subspace $\operatorname{span}\set{e_{1},\dots,e_{n}}$
of $L^2(0,1)$. Moreover, let $S_n$ denote the orthogonal projection onto the specific $n$-dimensional subspace $\operatorname{span}\set{u_{1},\dots,u_{n}}$ of the first $n$ singular functions of $A$.
Then we have that
\begin{equation} \label{eq:HM1}
 \sum_{i=n+1}^{\infty} \sigma^{2}_{i}(A) = \norm{A(I - S_{n})}{HS}^{2} \leq \norm{A(I - Q_{n})}{HS}^{2}= \sum_{i=n+1}^\infty \|Ae_i\|^2_{L^2(0,1)}\,.
\end{equation}
 \end{proposition}

The following  technical lemma and its proof can be found in \cite[Lemma~4]{HM22}.

 \begin{lemma}\label{lem:HM1}
 Let~$\{s_{i}\}_{i=1}^\infty$ be a non-increasing sequence of positive numbers, and let $\omega$ be a positive number. Suppose that there is a
  constant $K<\infty$ such that $\sum \limits_{i=n+1}^{\infty}
  s_{i}^{2}\leq K n^{-2\omega}$ for $n=1,2,\dots$. Then there also exists a
  constant $\hat K<\infty$ such that~$s_{i}^{2} \leq \hat K i^{-(1+2\omega)}$ for $i=1,2,\dots.$
 \end{lemma}
Furthermore, there is another useful property of orthogonal polynomials, shown in the proof of Lemma~7.4 in~\cite[Chapter 2, p.69--70]{Nik91}.
\begin{lemma}\label{lm:ortho}
Let $\{p_j\}_{j=0,1,\dots}$ be a system of orthogonal polynomials with respect to some measure $\nu$ and satisfy the orthogonality relation
\begin{equation}\label{eq:orthog}
\int p_j(t)p_k(t)d\nu(t)=h_j\delta_{jk},
\end{equation}
where the value $h_j$ depends on the index $j$ and $\delta_{jk}$ denotes the Kronecker delta, equal to $1$ if $j=k$ and to $0$ otherwise. Then, the polynomials $q_j$ defined by
\begin{equation}\label{eq:q}
q_j(t)=\int\frac{p_j(t)-p_j(\tau)}{t-\tau}\,d\nu(\tau)
\end{equation}
possess the degree of $j-1$. Furthermore, we have the identity
$$\frac{p_i(t)-p_i(\tau)}{t-\tau}=\sum_{j=0}^{i-1}\frac{p_j(\tau)q_i(\tau)-p_i(\tau)q_j(\tau)}{h_j}\,p_j(t).$$
\end{lemma}
Based on the above Proposition~\ref{pro:HM1} as well as on Lemmas~\ref{lem:HM1} and \ref{lm:ortho} we can now improve the inequality chain \eqref{eq:chain1} for updating the asymptotics of the singular values $\sigma_n(A)$
of the composite operator $A=C \circ J$ under consideration.

\medskip

\begin{theorem} \label{thm:improvedrate2}
For the composition $A=C \circ J: L^2(0,1) \to L^2(0,1)$ of the simple integration operator $J: L^2(0,1) \to L^2(0,1)$ from \eqref{eq:J} followed by the continuous Ces\`{a}ro operator $C: L^2(0,1) \to L^2(0,1)$ from \eqref{eq:C}, the asymptotics
\begin{equation} \label{eq:Arate}
\sigma_n(A) \asymp \frac{1}{n^2}
\end{equation}
holds true. In other words, the operator $A$ possesses exactly the degree {\it two} of ill-posedness in the sense of Definition~\ref{def:degree}.
\end{theorem}
\begin{proof}
We are going to apply Proposition~\ref{pro:HM1} for $A=C \circ J$ with the specific orthonormal basis $\{P_i\}_{i=1}^\infty$ in $L^2(0,1)$ of normalized shifted Legendre Polynomials, 
which are defined by means of the standard orthogonal Legendre Polynomials  $\{L_i(t)\}_{i=0}^\infty$ on $[-1,1]$ with the relationship $$P_{i+1}(t)=\sqrt{2i+1}\,L_i(2t-1)\,.$$
In this proof, we set $\tilde{t}:=2t-1$, $\tilde{s}:=2s-1$ for $t,s\in [0,1]$ and we will use the well-known three-term recurrence relation
\begin{equation}\label{eq:rekurs}
(i+1)\,L_{i+1}(\tilde{t})=(2i+1)\,\tilde{t}\,L_i(\tilde{t})-i\,L_{i-1}(\tilde{t}),
\end{equation}
 as well as the initial conditions
$$L_0(\tilde{t})=1,\qquad L_1(\tilde{t})=\tilde{t}.$$
 Firstly, we obtain
 $$[JP_{i+1}](s)=\int_0^s P_{i+1}(t)dt=\frac{L_{i+1}(\tilde{s})-L_{i-1}(\tilde{s})}{2\sqrt{2i+1}}.  $$
Then we define
$$f_i({t}):=\frac{L_i(\tilde{t})-L_i(-1)}{2t}=\frac{L_i(\tilde{t})-L_i(-1)}{\tilde{t}+1}$$ and verify
$$[CP_{i+1}](s)= \frac{1}{s}\int_0^s P_{i+1}(t)dt=\frac{f_{i+1}({s})-f_{i-1}({s})}{\sqrt{2i+1}}.$$

Applying Lemma \ref{lm:ortho} and identifying the orthogonal polynomials $\{p_j\}_{j=0,1,\dots}$ as $\{L_j\}_{j=0,1,\dots}$ on $[-1,1]$ such that  $h_j=\frac{2}{2j+1}$ and $L_j(-1)=(-1)^j$, one can check that $q_j(x)$ from \eqref{eq:q} satisfies the recursion relation \eqref{eq:rekurs}.
Moreover, the equation
$$q_j(-1)=2(-1)^{j-1}H_j$$
holds true, where $H_j:=\sum_{k=1}^j\frac{1}{k}$.
Consequently, we have
$$f_i({t})=\sum_{j=0}^{i-1}(-1)^{i+j-1}(2j+1)(H_i-H_j)L_j(\tilde{t})$$
and
$$c_{i}(s)=[CP_{i+1}](s)=\frac{(-1)^i\sqrt{2i+1}}{i(i+1)}\sum_{j=0}^{i-1}(-1)^j(2j+1)L_j(\tilde{s})+\frac{\sqrt{2i+1}}{i+1}L_i(\tilde{s}).$$

Finally, we calculate
$$a_i(s)=[AP_{i+1}](s)=[CJP_{i+1}](s)=\frac{\frac{c_{i+1}(s)}{\sqrt{2i+3}}-\frac{c_{i-1}({s})}{\sqrt{2i-1}}}{2\sqrt{2i+1}}$$
and obtain the complex expression
\begin{small}
\begin{equation*}
\begin{aligned}
a_i(s)&=\underbrace{(-1)^{i}\frac{\sqrt{2i+1}}{(i-1)i(i+1)(i+2)}}_{:=k_1}\sum_{j=0}^{i-2}(-1)^j(2j+1)L_j(\tilde{s})\\
&+\underbrace{\frac{i^2-4i-2}{2i(i+1)(i+2)\sqrt{2i+1}}}_{:=k_2}L_{i-1}(\tilde{s})\\
&-\underbrace{\frac{\sqrt{2i+1}}{2(i+1)(i+2)}}_{:=k_3}L_i(\tilde{s})+\underbrace{\frac{1}{2(i+2)\sqrt{2i+1}}}_{:=k_4}L_{i+1}(\tilde{s})\,.
\end{aligned}
\end{equation*}
\end{small}

{\parindent0em Since} all terms above are orthogonal, the squared $L^2$-norm for $a_i$ attains the form
\begin{small}
\begin{equation*}
\begin{aligned}
\|AP_{i+1}\|_{L^2(0,1)}^2&= \|a_i\|_{L^2(0,1)}^2=\int_0^1a_i^2(s)ds=\frac{1}{2}\int_{-1}^1 a_i^2(\tilde{s})d\tilde{s}\\
&=k_1^2\sum_{j=0}^{i-2}(2j+1)+k_2^2\frac{1}{2i-1}+k_3^2\frac{1}{2i+1}+k_4^2\frac{1}{2i+3}\\
&=\frac{2i+1}{i^2(i+1)^2(i+2)^2}+\frac{(i^2-4i-2)^2}{4i^2(i+1)^2(i+2)^2(2i+1)(2i-1)}\\
&+\frac{1}{4(i+1)^2(i+2)^2}+\frac{1}{4(i+2)^2(2i+1)(2i+3)}\\
&=\frac{3}{2i(i+1)(2i-1)(2i+3)}\,.\\
\end{aligned}
\end{equation*}
\end{small}
and
$$\|AP_i\|_{L^2(0,1)}^2=\frac{3}{2i(i-1)(2i-3)(2i+1)}.$$
Now we can apply Proposition \ref{pro:HM1} immediately and derive that
$$\sum_{i=n+1}^\infty\sigma_i^2(A)\leq \sum_{i=n+1}^\infty\|AP_i\|_{L^2(0,1)}^2=\frac{1}{8n^3-2n}\leq Kn^{-3}$$
for a constant $K<\infty$.
According to Lemma \ref{lem:HM1} by identifying $s_i$ as $\sigma_i(A)$, there exists a positive constant $\hat{K}$ such that $\sigma_i^2(A)\leq \hat{K}^2\,i^{-4} $ and consequently
$$\sigma_i(A)\leq \hat{K}i^{-2}.$$
Taking into account the estimates of \eqref{eq:chain1} with focus on the lower bound, this  shows the asymptotics
$$\sigma_n(A)\asymp \frac{1}{n^2}$$
and completes the proof of the theorem.
\end{proof}

It is interesting to notice that the set $\{P_i\}_{i=1}^\infty$ of the (shifted) Legendre polynomials as orthonormal basis seems to be sufficiently close to the eigensystem $\{u_i(A)\}_{i=1}^\infty$ as part of the singular system of the Hilbert-Schmidt operator $A$. As the following remark indicates, the corresponding eigensystem $\{u_i(J)\}_{i=1}^\infty$ from \eqref{eq:SVDJ} of the integration operator $J$ does not reach the best result to determine the upper bound of the rate $\sigma_n(A)$.

\begin{remark} \label{rem:improvedrate}
{\rm Applying Proposition~\ref{pro:HM1} for $A=C \circ J$ with the specific orthonormal basis $\{e_i\}_{i=1}^\infty$ in $L^2(0,1)$ of the form
$$e_i(t):=u_i(J)=\sqrt{2}\cos((i-\frac{1}{2})\pi t)\quad (0 \le t \le 1)$$
(see the singular system \eqref{eq:SVDJ}), it is only possible to obtain the degree of ill-posedness for operator $A$ as
\begin{equation} \label{eq:chain2}
\frac{\bar K_1}{n^2} \le \sigma_n(A) \le \frac{\bar K_2}{n^{3/2}}  \qquad (n \in \N)\,,
\end{equation}
here, where $\bar K_1$ and $\bar K_2$  are some positive constants. \hfill \fbox{}
}\end{remark}

\begin{remark} \label{rem:multi}
{\rm In \cite{HW05} it was shown that multiplication operators $M:L^2(0,1)\to L^2(0,1)$ from \eqref{eq:multioperator} with multiplier functions $m(t)=t^\eta\;(\eta>0)$ do not change the ill-posedness degree of $J$ when they occur in a composition $M \circ J$. From the present study, we can see  that this effect is also observable for the composition $J^2=M \circ (C \circ J)$ (cf.~\eqref{eq:Mcomp}). This means that such multiplication operator $M$ in the special case $\eta=1$ also does not amend the ill-posedness degree when moving from to $A=C\circ J$ to $J^2=M \circ A$. In coincidence, the following Figure \ref{fig:SVD}
\begin{figure}[H]
\begin{center}
\includegraphics[width=10cm]{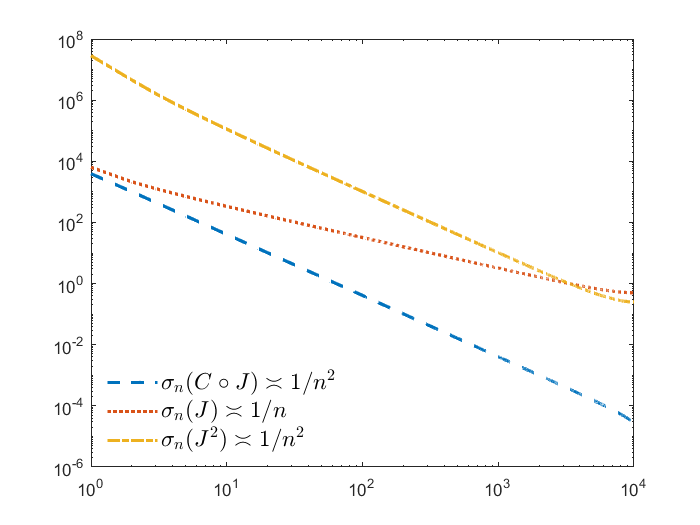}
\caption{Singular values of operators $A, J,$ and $J^2$}
\label{fig:SVD}
\end{center}
\end{figure}

{\parindent0em illustrates} the logarithmic plot of singular values of $\ell\times \ell$ discretization matrices with $\ell=10^4$ of the operators $A=C\circ J$, $J$ and $J^2$, which are calculated based on the MATLAB routine \textsc{svd}.\hfill \fbox{}
}\end{remark}

\section*{Acknowledgment}
YD and BH are supported by the German Science Foundation (DFG) under grant~HO~1454/13-1 (Project No.~453804957).

\section*{Appendix}

{\parindent0em {\bf Proof of Lemma~\ref{lem:two}:}}
To prove the non-compactness of $C$ it is enough to find a sequence $\{x_n\}_{n=1}^\infty$ in $L^2(0,1)$ such that $x_n \rightharpoonup 0$ (weak convergence in $L^2(0,1)$), but $\|Cx_n\|_{L^2(0,1)} \not \to 0$ as $n \to \infty$. In this context, we use
the bounded sequence $x_n(t)=\sqrt{n}\,\chi_{(0,\frac{1}{n}]}(t)\;(0 \le t \le 1)$ with $\|x_n\|_{L^2(0,1)}=1$ for all $n\in\N$. Then we have, for all $0<s \le 1$ and sufficiently large $n \in N$, that
\begin{equation*}
\int_0^s x_n(t)\,dt = \int_0^{1/n} \sqrt{n}\,dt= \frac{1}{\sqrt{n}}\,,
\end{equation*}
which tends to zero as $n \to \infty$. This shows (cf., e.g.,~\cite[Satz~10, p.~151]{Ljusternik68}) the claimed weak convergence. On the other hand, we have
\begin{equation*}
[C x_n](s)=\left\{\begin{array}{ll}
\sqrt{n}& \quad (0<s\leq \frac{1}{n})\\
\frac{1}{\sqrt{n}s}&\quad (0<\frac{1}{n}\leq s\leq 1)
\end{array} \right.\,.
\end{equation*}
Hence
$$\|C x_n\|_{L^2(0,1)}^2= \int_0^{1/n} n ds + \int_{1/n}^1 \frac{1}{n s^2} ds \to 2 \quad \text{as} \;\; n\to \infty\,,$$
and $C$ is not compact.

In order to prove the unboundedness of $C^{-1}$, we can exploit the sequence \linebreak $x_n(t)=\sqrt{n}\,\cos(nt)\;(0 \le t \le 1)$ together with the associated sequence
$y_n(s):=[Cx_n](s)=\frac{\sin(ns)}{\sqrt{n}s}\;(0 \le s \le 1)$ possessing the limit
$$\|y_n\|^2_{L^2(0,1)}= \|Cx_n\|^2_{L^2(0,1)}= \int_0^1\frac{\sin^2(ns)}{ns^2} \to \sqrt{\frac{\pi}{2}}< \infty \quad \text{as} \;\; n\to \infty\,. $$
Now the property
$$\|x_n\|^2_{L^2(0,1)}=\|C^{-1}y_n\|^2_{L^2(0,1)}=\int_0^1 n \cos^2(nt) dt=\frac{n+\sin(n)\cos(n)}{2} \to \infty  $$
for $n \to \infty$ indicates that $C^{-1}$ cannot be bounded, which completes the proof of the lemma. \hfill \fbox{}


\end{document}